\documentclass[11pt]{amsart}
\usepackage{fancyhdr}
\usepackage{amsmath}
\usepackage{bm}
\usepackage{hyperref}
\usepackage{graphicx} 
\usepackage{nicematrix}
\usepackage{amsthm,amssymb,amsmath}
\usepackage[T1]{fontenc}
\usepackage[utf8]{inputenc}
\usepackage{hyperref}
\usepackage{lipsum}
\usepackage{mathrsfs}
\usepackage{enumitem}
\usepackage{stmaryrd}
\usepackage{setspace}
\usepackage{yfonts}
\usepackage{float}
\usepackage{mathtools}
\usepackage{parskip}
\usepackage[all,cmtip]{xy}
\usepackage{tikz-cd}
\tikzcdset{row sep/normal=50pt, column sep/normal=50pt}

\newtheorem{lemma}{Lemma}[section]
\newtheorem{remark}[lemma]{Remark}
\newtheorem{theorem}[lemma]{Theorem}
\newtheorem{theorem*}{Theorem}

\newtheorem{corollary}[lemma]{Corollary}

\newtheorem{conjecture}[lemma]{Conjecture}
\newtheorem{manualtheoreminner}{Theorem}
\newenvironment{manualtheorem}[1]{%
  \renewcommand\themanualtheoreminner{#1}%
  \manualtheoreminner
}{\endmanualtheoreminner}
\usepackage{blindtext}
\usepackage{calligra}
\usepackage[a4paper, total={6in, 8in}]{geometry}
\usepackage[backend=biber, style=numeric, sorting=nyt]{biblatex}
\usepackage[title]{appendix}
\usepackage{xcolor}%
\usepackage{textcomp}%
\usepackage{manyfoot}%
\usepackage{booktabs}%
\usepackage{algorithm}%
\usepackage{algorithmicx}%
\usepackage{algpseudocode}%
\usepackage{listings}%

\DeclareMathOperator{\Ext}{\mathscr{E}\text{\kern -3pt {\calligra\large xt}}\,}
\DeclareMathOperator{\Hom}{\mathscr{H}\text{\kern -3pt {\calligra\large om}}\,}
\title[Toric image and amplified endomorphism]{Images of toric varieties and amplified endomorphism of weak Fano threefolds}
\author{Supravat Sarkar}
\date{}
\begin{document}

\begin{abstract}
    We show that some important classes of weak Fano $3$-folds of Picard rank $2$ do not have Bott vanishing. Using this we show that any smooth projective $3$-fold $X$ of Picard rank $2$ with $-K_X$ nef which is the image of a projective toric variety is toric. This proves a special case of a conjecture by Occhetta-Wi{\'s}niewski, extending a corresponding previous work for Fano $3$-folds. We also show that a weak Fano $3$-fold of Picard rank $2$ having an int-amplified endomorphism is toric. This proves a special case of a conjecture by Fakhruddin, Meng, Zhang and Zhong, extending corresponding previous work for Fano $3$-folds.
\end{abstract}
\maketitle
\begin{center} 
\end{center}
\begin{center}
\end{center}
\begin{center}
    \textbf{Author information}: Fine hall, Princeton University, ss6663@princeton.edu
\end{center}
\maketitle
\section{Acknowledgement}
I thank Burt Totaro, Jakub Witaszek and János Kollár for insightful discussions and valuable ideas.
\section{Introduction}
Unless otherwise stated, we work over the field of complex numbers. This article is devoted to proving special cases of two very related conjectures. These are the following:
\begin{conjecture}\label{1}
    If a smooth projective variety $X$ admits a surjective map $\phi:Z\to X$ from a proper toric variety $Z$, then $X$ is a toric variety.
\end{conjecture}
\begin{conjecture}\label{2}
    A smooth projective rationally connected variety admitting an int-amplified endomorphism is toric.
\end{conjecture}
 Conjecture \ref{1} is attributed to Occhetta and Wi{\'s}niewski, who proved it in the special case of $\rho(X)=1$ in \cite{occhetta2002euler}, where $\rho(X)$ denotes the Picard rank of $X$. One motivation of studying this question is the famous result of Lazarsfield that says the only smooth projective variety admitting a nonconstant surjective map from $\mathbb{P}^n$ is $\mathbb{P}^n$ itself. Conjecture \ref{1} is also known when $X$ is a surface or Fano $3$-fold by {\cite[Proof of Theorem 4.4.1]{achinger2021global}}, {\cite[Theorem 6.9, 7.7]{achinger2023global}} and {\cite[Theorem 7.2]{totaro2023endomorphisms}}. Similar questions have been studied in \cite{hwang2001projective} and \cite{demailly2008compact}.

 In this paper, we prove it for threefolds of Picard rank $2$ and nef anticanonical bundle.

 \begin{manualtheorem}{A}\label{A}
    Let $X$ be a smooth projective threefold with $\rho(X)=2$, and $-K_X$ is nef. Suppose $X$ is a toric image. Then $X$ is either toric Fano or one of the following: $$\mathbb{P}_{\mathbb{P}^1}(\mathcal{O}_{\mathbb{P}^1}\oplus\mathcal{O}_{\mathbb{P}^1}(1)^2),\mathbb{P}_{\mathbb{P}^1}(\mathcal{O}_{\mathbb{P}^1}^2\oplus\mathcal{O}_{\mathbb{P}^1}(2)) , \mathbb{P}_{\mathbb{P}^2}(\mathcal{O}_{\mathbb{P}^2}\oplus\mathcal{O}_{\mathbb{P}^2}(3)).$$
\end{manualtheorem}

Here \textit{toric image} means the image of a proper toric variety under a surjective map. Incidentally, it also gives a proof without using much toric geometry that the three varieties in the theorem are the only toric weak Fano threefolds\footnote{A smooth projective variety $X$ is weak Fano if $-K_X$ is nef and big.} of Picard rank $2$ which are not Fano. 

Conjecture \ref{2} is due to Fakhruddin, Meng, Zhang and Zhong, see {\cite[Question 4.4]{fakhruddin2002questions}}, \cite{meng2022non}. It was motivated by the study of the strong restrictions on a projective variety imposed by the existence of an int-amplified endomorphism. We will recall the definition of int-amplified endomorphism in \S 2. For smooth non-uniruled projective varieties, the structure is essentially known: these are the \'etale quotients of abelian varieties (see \cite{meng2017building}). For uniruled varieties, using an equivariant minimal model model program one can reduce to rationally connected varieties. So Conjecture \ref{2} is fundamental in the classification of projective varieties with int-amplified endomorphism. This conjecture is trivial for curves, and known for surfaces by \cite{nakayama2008complex}, projective hypersurfaces by \cite{beauville2001endomorphisms}, \cite{paranjape1989self} and Fano threefolds by {\cite[Theorem 1.4]{meng2022non}}, {\cite[Theorem 6.1]{totaro2023endomorphisms}}. 

 In this paper, we prove it for weak Fano threefolds of Picard rank $2$.

\begin{manualtheorem}{B}\label{B}
    Let $X$ be a smooth projective weak Fano threefold with $\rho(X)=2$. Suppose $X$ has an int-amplified endomorphism. Then $X$ is either toric Fano or one of the following: $$\mathbb{P}_{\mathbb{P}^1}(\mathcal{O}_{\mathbb{P}^1}\oplus\mathcal{O}_{\mathbb{P}^1}(1)^2),\mathbb{P}_{\mathbb{P}^1}(\mathcal{O}_{\mathbb{P}^1}^2\oplus\mathcal{O}_{\mathbb{P}^1}(2)) , \mathbb{P}_{\mathbb{P}^2}(\mathcal{O}_{\mathbb{P}^2}\oplus\mathcal{O}_{\mathbb{P}^2}(3)).$$
\end{manualtheorem}

One of the main tools in our proofs is the concept of Bott vanishing developed in \cite{kawakami2023endomorphisms}, \cite{totaro2023endomorphisms} and \cite{totaro2023bott}. In \cite{stark2024cone}, it is shown that $68$ families of Fano $3$-folds do not have Bott vanishing. In Corollary \ref{bott}, we show failure of Bott vanishing for several classes of weak Fano $3$-folds of Picard rank $2$. We use this result in our proofs of main theorems.

Before a few days I posted the first version of this article in arxiv, another preprint \cite{chen2025smooth} appeared in arxiv, proving that any rationally connected threefold of Picard rank $2$ admitting int-amplified endomorphism is toric. Their work is independent of me, they do not use the method of Bott vanishing, and as far as I know, their method does not prove my Theorem \ref{A}. On the other hand, my arguments for Theorem \ref{B} do not extend to prove the Theorem in \cite{chen2025smooth}. A particular feature of the arguments in this paper is that the proofs of Theorems \ref{A} and \ref{B} are very similar.

\section{Preliminaries}

\begin{enumerate}
    \item For a normal projective variety $X$ and distinct prime divisors $D_i$ in $X$, we call $(X,\sum_i D_i)$ \textit{toric image} if there is a proper toric variety $Z$ and a surjective morphism $f:Z\to X$ such that the irreducible components of each $f^{-1}(D_i)$, which are of codimension $1$ in $Z$, are toric divisors. In this case, using Stein factorization and Blanchard's theorem (Lemma \ref{Blanchard}), one sees that there is $(Z, f)$ with the above properties with $f$ a finite map. We say $X$ is a toric image if the pair $(X,0)$ is a toric image.
    \item We say an endomorphism $f$ of a normal projective variety $X$ is \textit{int-amplified} if there is an ample line bundle $H$ on $X$ such that $f^*H\otimes H^{-1}$ is ample. This implies in particular that $f$ is finite surjective.
    \item For a normal projective variety $X$ and distinct prime divisors $D_i$ in $X$, we say $(X,\sum_i D_i)$ has an int-amplified endomorphism if there is an int-amplified endomorphism $f$ of $X$ with $f^{-1}(D_i)=D_i$ set-theoretically for all $i$.
    \item Recall the notion of log Bott vanishing introduced in \cite{totaro2023endomorphisms}. For a normal projective variety $X$ and reduced Weil divisor $D$ on $X$, we say $(X, D)$ has \textit{log Bott vanishing} if
    \[
H^j\!\left(
X,\,
\Omega_X^{[i]}(\log D)(A-E)
\right)
=0
\]
for every reduced divisor \(0\le E\le D_X\), \(i\ge 0\), \(j>0\), and \(A\) an ample Weil divisor on \(X\). Here $\Omega_X^{[i]}(\log D)$ is the reflexivization (double dual) of the sheaf $\Omega_X^{i}(\log D)$ of $i$-forms with logarithmic poles along $D$. We say $X$ has Bott vanishing if $(X, 0)$ has log Bott vanishing.
     \item  Let $A^{\ast}(X)$ denote the Chow ring of a smooth projective variety $X$. We identify Pic$(X)=A^1(X)$, where by Pic(X) we mean the Picard group of $X$. For a vector bundle $E$ over a smooth projective variety $X$, the total Chern class of $E$ is denoted by $c(E)\in A^{\ast}(X)$.  For a vector bundle $E$ on $\mathbb{P}^n$, there are unique integers $c_i$ such that $c_i(E)=c_ih^i$, where $h\in A^1(\mathbb{P}^n)$ is the class of a hyperplane. By abuse of notation, we will write $c_i(E)=c_i.$

\item A morphism $f:X\to Y$ of normal projective varieties is called a \textit{contraction} if $f_*\mathcal{O}_X=\mathcal{O}_Y$. Also, $f$ is called a \textit{small} contraction if $f$ is a birational contraction and the exceptional locus of $f$ have codimension $\geq 2$ in $X$.
\item We will follow the notation and convention of projective bundles as in \cite[Chapter II, Section 7]{hartshorne2013algebraic}.
\end{enumerate}

\section{Failure of Bott vanishing}

The following theorem is crucial in our proof of the main theorems.
\begin{theorem}\label{rr}
The following statements hold.
    \begin{enumerate}
        \item Let $X$ be a smooth projective weak Fano $3$-fold with $\rho(X)=2$. Let $\phi:X\to Y$, $\psi: X\to X'$ be the contractions of the two rays of $\overline{NE}(X)$, $\phi$ being a $K_X$-negative contraction. Let $H$ be the pullback of the ample generator of $\mathrm{Pic}(Y)$ to $X$, considered as an element of $A^1(X)$. Let $h=h^2(X,\Omega^1_X)$, $c_i=c_i(T_X)$ for $1\leq i\leq 3.$ Then we have 
        $$-\chi(X, \Omega^2_X(H-K_X))=16+h-\frac{c_1^3}{2}-\frac{5}{4}(c_1^2H+c_1H^2)+\frac{3}{4}c_2H-\frac{1}{2}H^3.$$
        \item Let $a_0, a_1, a_2, a_3$ be integers, not all distinct. Let $\mathcal{E}=\oplus_{i=0}^3\mathcal{O}_{\mathbb{P}^1}(a_i)$, a vector bundle of rank $4$ over $\mathbb{P}^1$. Let $U=\mathcal{O}_{\mathbb{P}(\mathcal{E})}(1)$, $H$ be the pullback of $\mathcal{O}_{\mathbb{P}^1}(1)$ to $\mathbb{P}(\mathcal{E}).$ Let $k$ be an integer, and $X$ be a smooth irreducible member of the complete linear system $|kH+2U|$ on $\mathbb{P}(\mathcal{E}).$ Then for every integer $a,$ we have
        $$-\chi(X, \Omega^2_X(aH+U))=2\big(\sum_i a_i+2k\big).$$ Here by abuse of notation, $H, U$ denotes the restrictions of $H, U$ to $X.$
        \item Let $\mathcal{E}$ be a vector bundle of rank $2$ on $\mathbb{P}^2$ with $\mathcal{E}(1)$ ample. Let $X=\mathbb{P}(\mathcal{E})$, $U=\mathcal{O}_{\mathbb{P}(\mathcal{E})}(1)$, $H$ be the pullback of $\mathcal{O}_{\mathbb{P}^2}(1)$ to $X.$ For $i=1,2$, let $c_i$ be the integer representing $c_i(\mathcal{E})$. Then
        $$-\chi(X, \Omega^2_X(H+U))=c_2-{c_1\choose{2}},$$ and $$h^0(X,\Omega^2_X(H+U))=h^2(\mathbb{P}^2,\mathcal{E}(-c_1-1)). $$
       
        \end{enumerate}
\end{theorem}
\begin{proof}   In what follows, we prove statements (1), (2) and (3) of the theorem.

    \underline{$(1)$}: We will use the Hirzebruch-Riemann-Roch theorem for $3$-folds. It says the following. Let $X$ be a smooth projective $3$-fold, $E$ a vector bundle on $X$ of rank $r$. Then we have
    \begin{equation}\label{hrr}
        \chi(X, E)=\frac{rc_1c_2}{24}+e_1\cdot\frac{c_1^2+c_2}{12}+\frac{c_1}{2}\cdot\frac{e_1^2-2e_2}{2}+\frac{e_1^3-3e_1e_2+3e_3}{6},
    \end{equation} (see {\cite[Proof of Theorem 3.1.1]{stark2024cone}}). Here $c_i=c_i(T_X),$ $e_i=c_i(E)$.

    In our setup, we have by Serre duality 
    \begin{equation}\label{Serre1}
        -\chi(X, \Omega^2_X(H-K_X))=\chi(X,\Omega_X(-H+K_X)).
    \end{equation} Let $E=\Omega_X(-H+K_X),$ $e_i=c_i(E), c_i=c_i(T_X).$ By Kawamata–Viehweg vanishing, $H^i(X,\mathcal{O}_X)=0$ for all $i>0.$ So, $c_3=$ topological Euler characteristics of $X=6-2h$, using $\mathrm{Pic}(X)\cong H^2(X, \mathbb{Z})\cong \mathbb{Z}^2.$ Also, $c_1c_2=24$ by putting $E=\mathcal{O}_X$ in \eqref{hrr}.

    A calculation shows
    \begin{align*}
        & e_1=-4c_1-3H, \\ &e_2=c_2+5c_1^2+8c_1H+3H^2,\\ & e_3=-(2c_1^3-2h+30+c_2H+5c_1^2H+4c_1H^2+H^3).
    \end{align*}

    Putting these values in \eqref{hrr} and using \eqref{Serre1}, we get the result.

    \underline{$(2)$}: Tensoring $\mathcal{E}$ by a line bundle if necessary, we may assume without loss of generality that two of the $a_i$'s are $0$. Let the other two be $p$ and $q$. So, $\mathcal{E}=\mathcal{O}_{\mathbb{P}^1}^2\oplus \mathcal{O}_{\mathbb{P}^1}(p)\oplus \mathcal{O}_{\mathbb{P}^1}(q).$ By Serre duality,
    \begin{equation}\label{Serre2}
        -\chi(X, \Omega^2_X(aH+U))=\chi(X,\Omega_X(-aH-U)).
    \end{equation}

    Let $W=\mathbb{P}(\mathcal{E})$. As in the proof of {\cite[Lemma 6.5]{totaro2023endomorphisms}}, we have short exact sequences:
    \begin{equation}
        0\to \mathcal{O}_X(-aH-U-X)\to \Omega_W(-aH-U)|_X\to  \Omega_X(-aH-U)\to 0,
    \end{equation}
    \begin{equation}
        0\to \Omega_W(-aH-U-X)\to\Omega_W(-aH-U) \to \Omega_W(-aH-U)|_X \to 0,
    \end{equation}
    and
    \begin{equation}
        0\to \mathcal{O}_W(-aH-U-2X)\to\mathcal{O}_W(-aH-U-X)\to \mathcal{O}_X(-aH-U-X)\to 0.
    \end{equation}

    Together with $\mathcal{O}_W(X)=kH+2U$, these give:
    \begin{equation}\label{chi}
    \begin{aligned}
        \chi(X,\Omega_X(-aH-U)) &=\chi(W,\Omega_W(-aH-U))-\chi(W,\Omega_X(-(a+k)H-3U))\\ & -\chi(W,-(a+k)H-3U)+\chi(W,-(a+2k)H-5U)).
    \end{aligned}
    \end{equation}

    For integers $x$ and $y$, let $$f(x,y)=\chi(W, xH+yU).$$ Note that $W$ is a smooth projective toric variety, whose toric boundary has $6$ irreducible components, linearly equivalent to $H, H, U, U, U-pH, U-qH$, respectively. By Euler-Jaczewsky sequence as in \cite{occhetta2002euler}, we have a short exact sequence
    \begin{equation}
        0\to \Omega_W\to \mathcal{O}_W(-H)^2\oplus \mathcal{O}_W(-U)^2\oplus \mathcal{O}_W(pH-U)\oplus \mathcal{O}_W(qH-U)\to \mathcal{O}_W^2\to 0.
    \end{equation}

    This shows 
    \begin{equation}
        \chi(W,\Omega_W( xH+yU))=2f(x-1,y)+2f(x,y-1)+f(x+p,y-1)+f(x+q,y-1)-2f(x,y),
    \end{equation} for integers $x$ and $y.$

Plugging it into \eqref{chi}, we get
\begin{equation}\label{chi f}
    \begin{aligned}
        \chi(X,\Omega_X(-aH-U)) &= 2f(-a-1,-1)+2f(-a,-2)-2f(-a,-1) \\
        &+ f(-a+p, -2) + f(-a+q, -2) \\
        &-2f(-a-k-1, -3)-2f(-a-k, -4)+f(-a-k,-3) \\
        &-f(-a-k+p,-4)-f(-a-k+q,-4)\\
        &+f(-a-2k,-5).
    \end{aligned}
\end{equation}

Now we find a formula for $f(x, y).$ Let $p:W\to \mathbb{P}^1$ be the projection. Note that for $y\gg 0$, we have $R^ip_*\mathcal{O}_W(xH+yU)=0$ for $i>0,$ and $p_*\mathcal{O}_W(xH+yU)=S^y(\mathcal{E})(x)$. So, by spectral sequence we have for all $i$,
$$H^i(W, xH+yU)=H^i(\mathbb{P}^1, S^y(\mathcal{E})(x))$$ for $y \gg 0.$ Hence, 
\begin{equation}\label{f}
    f(x,y)=\chi(\mathbb{P}^1, S^y(\mathcal{E})(x)),
\end{equation} for $y\gg 0$.

We have rank $S^y(\mathcal{E})={y+3\choose 3}$, $\deg S^y(\mathcal{E})=\deg \mathcal{E}. {y+3\choose 4}=(p+q){y+3\choose 4}$. So, by Riemann-Roch, 
\begin{align*}
    \chi(\mathbb{P}^1, S^y(\mathcal{E})(x))&=\deg S^y(\mathcal{E})(x)+ \mathrm{rank} S^y(\mathcal{E})(x)\\ &= \deg S^y(\mathcal{E})+ (x+1)\mathrm{rank} S^y(\mathcal{E})\\ &=(p+q){y+3\choose 4}+(x+1){y+3\choose 3}.
\end{align*} 

So by \eqref{f},
\begin{equation}\label{f formula}
    \begin{aligned}
    f(x,y)&=\deg S^y(\mathcal{E})(x)+ \mathrm{rank} S^y(\mathcal{E})(x)\\ &= \deg S^y(\mathcal{E})+ (x+1)\mathrm{ rank} S^y(\mathcal{E})\\ &=(p+q){y+3\choose 4}+(x+1){y+3\choose 3},
\end{aligned} 
\end{equation} whenever $y\gg0.$ Since both sides of \eqref{f formula} are polynomials in $x,y$ by Riemann-Roch, \eqref{f} holds for every pair of integers $x$ and $y.$

By \eqref{f formula}, we get $f(x,y)=0$ if $-3\leq y<0.$ Plugging \eqref{f formula} into \eqref{chi f}, we get

\begin{align*}
    \chi(X,\Omega_X(-aH-U)) &=-2f(-a-k,-4)-f(-a-k+p,-4)\\ &-f(-a-k+q,-4)+f(-a-2k,-5)\\ &=(p+q).(-2{-1\choose4}-{-1\choose4}-{-1\choose4}+{-2\choose4})\\ & +2k{-1\choose3}+(k-p){-1\choose 3} +(k-q){-1\choose3} -2k{-2\choose3}\\
    & +(1-a)(-2{-1\choose 3}-{-1\choose3}-{-1\choose3}+{-2\choose3})\\
    &=2(p+q+2k).
\end{align*}
Now by \eqref{Serre2}, we get the result.

\underline{$(3)$}: By Serre duality,
    \begin{equation}\label{Serre3}
        -\chi(X, \Omega^2_X(H+U))=\chi(X,\Omega_X(-H-U)).
    \end{equation} 
    
    Let $f:X\to \mathbb{P}^2$ be the projection. For an integer $b$, let $Q(b)=\chi(X,\Omega_X(-H+bU)).$ By Hirzebruch-Riemann-Roch theorem, $Q(b)$ is a polynomial in $b$. We
    have short exact sequences:
    \begin{equation}\label{omega}
        0\to f^*\Omega_{\mathbb{P}^2}(-H+bU)\to \Omega_X(-H+bU)\to \Omega_{X/\mathbb{P}^2}(-H+bU))\to 0,
    \end{equation}
    \begin{equation}\label{euler}
        0\to \Omega_{X/\mathbb{P}^2}(-H+bU))\to (f^*\mathcal{E})(-H+(b-1)U)\to \mathcal{O}_X(-H+bU)\to 0.
    \end{equation}
By {\cite[Exercise 8.4, Chapter 3]{hartshorne2013algebraic}}, we have $R^1f_*\mathcal{O}_X(-2U)=\mathcal{O}_{\mathbb{P}^2}(-c_1)$ and $R^if_*\mathcal{O}_X(-2U)=0$ for all $i$. So, putting $b=-1$ in \eqref{omega} and \eqref{euler} and applying $f_*$ we get isomorphisms $$R^1f_*\Omega_X(-H-U)\cong R^1f_*\Omega_{X/\mathbb{P}^2}(-H-U)\cong \mathcal{E}(-c_1-1).$$ Now by Serre duality and spectral sequence, we get $$h^0(X,\Omega^2_X(H+U))=h^3(X,\Omega_X(-H-U))=h^2(\mathbb{P}^2,R^1f_*\Omega_X(-H-U))=h^2(\mathbb{P}^2,\mathcal{E}(-c_1-1)).$$ This proves the second statement of part $3$ of the theorem.

Now we proceed to prove the first statement of part $3$ of the theorem.
    As $U$ is $f$-ample, we have 
    \begin{equation}\label{rel ample}
        R^if_*\Omega_X(-H+bU)=R^if_*\Omega_{X/\mathbb{P}^2}(-H+bU)=0
  \end{equation} for $i>0$ and $b\gg0$. So, using spectral sequence we have for all $i,$ $$H^i(X,\Omega_X(-H+bU))=H^i(\mathbb{P}^2, f_* \Omega_X(-H+bU)) $$ for $b\gg 0.$ Hence,
    \begin{equation}
        Q(b)= \chi(\mathbb{P}^2, f_* \Omega_X(-H+bU))
    \end{equation} for $b\gg 0.$

    Using \eqref{rel ample} and applying $f_*$ to \eqref{omega} and \eqref{euler}, we have short exact sequences for $b\gg 0$:
    \begin{equation}
        0\to \Omega_{\mathbb{P}^2}(-1)\otimes S^b\mathcal{E}\to f_*\Omega_X(-H+bU)\to f_*\Omega_{X/\mathbb{P}^2}(-H+bU)\to 0,
    \end{equation}
    \begin{equation}
        0\to f_*\Omega_{X/\mathbb{P}^2}(-H+bU)\to \mathcal{E}(-1)\otimes S^{b-1}\mathcal{E}\to (S^b\mathcal{E})(-1)\to 0.
    \end{equation}

    Hence, for $b\gg 0$ we have
    \begin{equation}\label{Q}
        Q(b)=\chi(\mathbb{P}^2, \Omega_{\mathbb{P}^2}(-1)\otimes S^b\mathcal{E})+ \chi(\mathbb{P}^2,\mathcal{E}(-1)\otimes S^{b-1}\mathcal{E} )-\chi(\mathbb{P}^2, (S^b\mathcal{E})(-1) ).
    \end{equation}

    There are polynomials $Q_1(u), Q_2(u), Q_3(u), C_1(u), C_2(u), A_1(u), A_2(u)$ such that for all $b\gg 0$ we have:
    \begin{align*}
        & Q_1(b)=\chi(\mathbb{P}^2, \Omega_{\mathbb{P}^2}\otimes (S^b\mathcal{E})(-1)),\\ 
        & Q_2(b)=\chi(\mathbb{P}^2,\mathcal{E}\otimes (S^{b-1}\mathcal{E})(-1) ),\\
        & Q_3(b)= \chi(\mathbb{P}^2, (S^b\mathcal{E})(-1) ),\\ 
        & C_1(b)=c_1(S^b\mathcal{E}),\\
        & C_2(b)=c_2(S^b\mathcal{E}), \\
        & A_1(b)= c_1((S^b\mathcal{E})(-1)),\\
        & A_2(b)= c_2((S^b\mathcal{E})(-1)).
    \end{align*}

    By \eqref{Q}, we have 
    \begin{equation}\label{Qsum}
        Q(u)=Q_1(u)+Q_2(u)-Q_3(u).
    \end{equation}

One sees:
\begin{equation}\label{sbe}
    \begin{aligned}
        &C_1(b)=c_1b(b+1)/2,\\
        & C_2(b)=c_1^2b(b^2-1)(3b+2)/24 +c_2{b+2\choose 3},
    \end{aligned}
\end{equation}
for all $b\gg 0.$ Since both sides of \eqref{sbe} are polynomials in $b$, \eqref{sbe} holds for all integer $b.$ In particular, we have 
\begin{equation}\label{c-1-2}
    \begin{aligned}
        & C_1(-1)=C_2(-1)=0,\\
        &C_1(-2)=c_1, C_2(-2)=c_1^2.
    \end{aligned}
\end{equation}

    To compute the polynomials $Q_i$'s, we will use Hirzebruch-Riemann-Roch theorem for vector bundles on $\mathbb{P}^2.$ It says the following. Let $\mathcal{F}$ be a vector bundle of rank $r$ on $\mathbb{P}^2$ with $c_i=c_i(\mathcal{F})$ for $i=1,2.$ Then we have
    \begin{equation}\label{hrr1}
        \chi(\mathbb{P}^2, \mathcal{F})=r-c_2+c_1(c_1+3)/2.
    \end{equation}

    We will also need the following formula of Chern classes of a tensor product of vector bundles on $\mathbb{P}^2$. Let $\mathcal{F}_1,\mathcal{F}_2 $ be vector bundles on $\mathbb{P}^2$ of ranks $r$ and $s$ respectively, and suppose $c_i(\mathcal{F}_1)=c_i, c_i(\mathcal{F}_2)=d_i$ for $i=1,2.$ Then we have:
    \begin{equation}\label{chernproduct}
        \begin{aligned}    &c_1(\mathcal{F}_1\otimes\mathcal{F}_2)=sc_1+rd_1,\\
  & c_2(\mathcal{F}_1\otimes\mathcal{F}_2)=sc_2+rd_2+{s\choose2} c_1^2+{r\choose 2}d_1^2+c_1d_1(rs-1).
        \end{aligned}
    \end{equation}
 Since rank $S^b(\mathcal{E})=b+1,$
    these give:
    \begin{equation}\label{AC}
    \begin{aligned}
        &A_1(u)=C_1(u)-(u+1),\\
        &A_2(u)=C_2(u) +{u+1\choose 2}-uC_1(u).
    \end{aligned}
    \end{equation}

So, \begin{equation}\label{a-1-2}
    \begin{aligned}
        &A_1(-1)=A_2(-1)=0,\\
        &A_1(-2)=c_1+1, A_2(-2)=(c_1+1)^2.
    \end{aligned}
\end{equation}
    Since $c_1(\Omega_{\mathbb{P}^2})=-3$ and $c_2(\Omega_{\mathbb{P}^2})=3$, \eqref{chernproduct} gives:
    \begin{equation}\label{ct}
        \begin{aligned}
            & c_1(\Omega_{\mathbb{P}^2}\otimes (S^b\mathcal{E})(-1))=2A_1(b)-3(b+1),\\
            &c_2(\Omega_{\mathbb{P}^2}\otimes (S^b\mathcal{E})(-1))=2A_2(b)+3(b+1)+A_1(b)^2+9{b+1\choose 2}-3(2b+1)A_1(b),\\
            & c_1(\mathcal{E}\otimes (S^{b-1}\mathcal{E})(-1) )=2A_1(b-1)+bc_1,\\
            & c_2(\mathcal{E}\otimes (S^{b-1}\mathcal{E})(-1) )=2A_2(b-1)+bc_2+A_1(b-1)^2+{b\choose 2}c_1^2 +(2b-1)c_1A_1(b-1),
        \end{aligned}
    \end{equation}
     for all $b\gg 0.$

     So, by \eqref{hrr1}, we get for all $b\gg 0$,
     \begin{equation}\label{Q1}
         \begin{aligned}
             Q_1(b)=&2(b+1)-(2A_2(b)+3(b+1)+A_1(b)^2+9{b+1\choose 2}-3(2b+1)A_1(b))\\
             &+(2A_1(b)-3(b+1))(2A_1(b)-3b)/2,
         \end{aligned}
    \end{equation} and
    \begin{equation}\label{Q2}
         \begin{aligned}
             Q_2(b)=&2b-(2A_2(b-1)+bc_2+A_1(b-1)^2+{b\choose 2}c_1^2 +(2b-1)c_1A_1(b-1))\\
             &+(2A_1(b-1)+bc_1)(2A_1(b-1)+bc_1+3)/2,
         \end{aligned}
    \end{equation} and 
    \begin{equation}\label{Q3}
     Q_3(b)=(b+1)-A_2(b)+A_1(b)(A_1(b)+3)/2.
    \end{equation}

    Since both sides of \eqref{Q1}, \eqref{Q2} and \eqref{Q3} are polynomials in $b$, we see that these equations are valid for all integers $b.$ So by \eqref{a-1-2} we get
    \begin{align}
        &Q_1(-1)=Q_3(-1)=0,\\
        & Q_2(-1)=c_2-{c_1\choose{2}}.
    \end{align}

    Now by \eqref{Qsum}, we have $Q(-1)=c_2-{c_1\choose{2}}.$ Finally, \eqref{Serre3} completes the proof.
\end{proof}
   \begin{corollary}\label{bott}
       Let $X$ be a smooth projective weak Fano $3$-fold with $\rho(X)=2$, and suppose $X$ is not Fano. If $X$ is as in the following cases, then $X$ does not have Bott vanishing.
       \begin{enumerate}
           \item $X$ is a degree $6$ del Pezzo fibration over $\mathbb{P}^1$.
           \item $X$ is a degree $8$ del Pezzo fibration over $\mathbb{P}^1$.
           \item $X$ is a conic bundle over $\mathbb{P}^2$ and $X\not\cong\mathbb{P}_{\mathbb{P}^2}(\mathcal{O}_{\mathbb{P}^2}\oplus \mathcal{O}_{\mathbb{P}^2}(3)).$ Also, $X$ is not as in {\cite[No. 1, Table 7.7]{jahnke2011threefolds}}.
           \item $X$ is as in {\cite[Table 8]{cutrone2013towards}}.
           \item $X$ is as in {\cite[Table 9]{cutrone2013towards}}.
           \item $X$ is as in {\cite[No.1, Table 7.5]{jahnke2004threefolds}}.
       \end{enumerate}
   \end{corollary} 
   \begin{proof}

   Note that if $X$ satisfies Bott vanishing, then for all ample line bundles $L$ on $X$, we have $\chi(X, \Omega^2_X\otimes L)\geq 0$. Also, if $\chi(X, \Omega^2_X\otimes L)=0$, then we must have $h^0(X, \Omega^2_X\otimes L)=0$. This will be the basis of our argument in each case.
   
       \underline{$(1):$} Let $F$ be a general fibre of the del Pezzo fibration. In the notation of \ref{rr}($1$), we have 
       \begin{align*}
           &c_1H^2=H^3=0,\\
           &  c_1^2H=K_X^2F=K_F^2=6,\\
           & c_2H=c_2(T_X|_F)=c_2(T_F)=\text{topological Euler characteristics of }F=6.
       \end{align*} So, \ref{rr}($1$) gives
       $$-\chi(X, \Omega^2_X(H-K_X))=13+h-\frac{c_1^3}{2}.$$ Now the result follows by looking at {\cite[Table 1]{fukuoka2019refinement}}.
       
    \underline{$(2):$}  If the anticanonical contraction of $X$ is a small contraction, $X$ is as in {\cite[Theorem 2.3]{takeuchi2022weak}}. Now Theorem \ref{rr}($2$) shows $X$ does not have Bott vanishing.

    Now suppose the anticanonical contraction of $X$ is a divisorial contraction. By {\cite[Section 4]{takeuchi2022weak}}, $X$ is one of the following:
    \begin{enumerate}
        \item [$(i):$] {\cite[(4.3.3)]{takeuchi2022weak}} ,
        \item [$(ii):$] {\cite[(4.3.6)]{takeuchi2022weak}},
        \item [$(iii):$] {\cite[(4.3.7)]{takeuchi2022weak}},
        \item [$(iv):$] {\cite[No. 11, Table 1]{takeuchi2022weak}}.
    \end{enumerate} Also, as in the numbering of {\cite[Table 7.1]{jahnke2004threefolds}}, $X$ is one of the following:
    \begin{enumerate}
        \item[$(i)':$] No. 9,
        \item [$(ii)':$]No. 12,
        \item [$(iii)':$] No. 14,
        \item[$(iv)':$] No. 15.
    \end{enumerate}

    Now a comparison using $-K_X^3$ and {\cite[Theorem 2.9]{jahnke2004threefolds}} shows $$(i)=(iii)', (ii)=(iv)', (iii)=(i)', (iv)=(ii)'.$$

    From this and the value of $-K_X^3$ obtained from {\cite[Table 7.1]{jahnke2004threefolds}} gives the values of $k$:
    \begin{align*}
        &(i):k=0,\\
        &(ii): k=0,\\
        &(iii): k=-1,\\
        &(iv): k=2.
        \end{align*}
        Now Theorem \ref{rr}($2$) shows $X$ does not have Bott vanishing.

        \underline{$(3):$} Let us consider two cases separately.
        
        \textbf{Case 1:} $X$ is a $\mathbb{P}^1$-bundle.
        
        In this case, $X$ is as in {\cite[No. 2,3,4, Table A.3]{jahnke2004threefolds}}, or {\cite[Table 7.5]{jahnke2011threefolds}}, or {\cite[No. 1, Table 7.6]{jahnke2011threefolds}} or as in $X^+$ of {\cite[No. 1, Table 7.2]{jahnke2011threefolds}}. In the first $3$ cases, $\mathcal{E}:=\mathcal{F}(2)$ is nef but not ample by the proof of {\cite[Theorem 3.4]{jahnke2004threefolds}}, in the other $3$ cases $\mathcal{E}:=\mathcal{F}$ or $\mathcal{F}^+$ is nef but not ample by {\cite[Theorem 2.13]{jahnke2011threefolds}}. 

        Now by Theorem \ref{rr}$(3)$, $X$ does not have Bott vanishing except possibly in the first case, where we have $-\chi(X, \Omega^2_X(H+U))=0$. But if $X$ has Bott vanishing in this case, we must have $h^0(X, \Omega^2_X(H+U))=0,$ hence $h^2(\mathbb{P}^2,\mathcal{E}(-4))=0.$ But a computation using the description of $\mathcal{F}$ by the short exact sequence in {\cite[Theorem 3.4(2)]{jahnke2004threefolds}} shows $h^2(\mathbb{P}^2,\mathcal{E}(-4))=1,$ a contradiction. So, $X$ does not have Bott vanishing.

        \textbf{Case 2:} $X$ is a conic bundle, but not a $\mathbb{P}^1$-bundle.

        We use the notation of Theorem \ref{rr}$(1)$. We have $Y=\mathbb{P}^2$ and $\phi$ a conic bundle. Let $d>0$ be the degree of the discriminant locus of $\phi.$ Let $L$ be a general line in $\mathbb{P}^2$, $F=\phi^{-1}(L)$. So, $F$ is the blow up of a Hirzebruch surface at $d$ points. A computation similar to the proof of part $(1)$ shows 
        \begin{align*}
           & c_1^2H=12-d,\\
           & c_1H^2=2,\\
           & c_2H=d+6,\\
           & H^3=0.
        \end{align*}
So, \ref{rr}($1$) gives
       $$-\chi(X, \Omega^2_X(H-K_X))= h-\frac{c_1^3}{2}+2d+3.$$

       Now the result follows by looking at {\cite[Table A.3]{jahnke2004threefolds}}, {\cite[Tables 7.2, 7.6, 7.7]{jahnke2011threefolds}}.
        
        \underline{$(4):$} Let $E$ be the exceptional divisor of the $K_X$-negative contraction of $X$. In the notation of Theorem \ref{rr}$(1)$, we have 
        \begin{equation}\label{discrep}
            -K_X=H-E.
        \end{equation}We have $c_2.E=c_2(T_X|_E)=0,$ $c_1^2.E=K^2E=2.$ Using \eqref{discrep} and the numerics in {\cite[Table 8]{cutrone2013towards}}, we get the following numerics for each case in the Table:
        
        No. 1: $c_1^3=4, c_1^2H=H^3=c_1H^2=6, c_2H=24$,

        No. 2: $c_1^3=2, c_1^2H=H^3=c_1H^2=4, c_2H=24$.

        Now Theorem \ref{rr}$(1)$ shows $X$ does not have Bott vanishing.
        
        \underline{$(5):$} Let $E$ be the exceptional divisor of the $K_X$-negative contraction of $X$. In the notation of Theorem \ref{rr}$(1)$, we have 
        \begin{equation}\label{discrep1}
            H=E-2K_X.
        \end{equation}We have $c_2.E=c_2(T_X|_E)=-3,$ $c_1^2.E=K^2E=1.$ Using \eqref{discrep1} and the numerics in {\cite[Table 9]{cutrone2013towards}}, we get the following numerics:
$$c_1^3=2, c_1^2H=5, H^3=20, c_1H^2=10, c_2H=45.$$
Now Theorem \ref{rr}$(1)$ shows $X$ does not have Bott vanishing.

        \underline{$(6):$} Proof is identical to the proof of the $4$-th part of the corollary.
   \end{proof}
\section{Images of toric varieties}
In this section we will prove Theorem \ref{A}. First we need the following lemmas.
\begin{lemma}\label{push}
    Let $f:X\to Y$ be a surjective morphism of normal projective varieties. Let $D_X$ be a sum of prime divisors in $X$ such that $(X, D_X)$ is toric image (see (2) in Preliminaries). Let $D_Y$ be a sum of prime divisors in $Y$ such that $D_Y$ is Cartier. Suppose $f^{-1}(D_Y)_{\rm red}\leq D_X$. Then $(Y, D_Y)$ is toric image.
\end{lemma}
\begin{proof}
    Immediate from the definition.
\end{proof}

\begin{lemma}\label{Blanchard}
    Let $X$ and $Y$ be normal projective varieties and $f:X\to Y$ a contraction. If $X$ is a toric variety, then $Y$ has a toric variety structure such that $f$ is a toric morphism.
\end{lemma}
\begin{proof}
    Follows from \cite[Corollary 1, Page 2]{occhetta2002euler}.
\end{proof}
\begin{lemma} \label{lemma:ex}
Let $f:X\to Y$ be a birational morphism of normal projective varieties and let $E\subset X$ be a prime divisor which is $f$-exceptional. Suppose $Z$ is a proper toric variety with a generically finite surjective map $g:Z\to X$. Then the divisorial part of $g^{-1}(E)_{red}$ is a toric divisor.
\end{lemma}
\begin{proof}
   The argument is similar to parts of the proof {\cite[Theorem 7.2]{totaro2023endomorphisms}}. Let $Z\xrightarrow{p}Y_1\to Y$ be the Stein factorization of $f\circ g.$ So we have a commutative diagram
    \begin{center}
\begin{tikzcd}
Z \arrow[r, "g"] \arrow[d,"p"]
& X \arrow[d, "f" ] \\
Y_1  \arrow[r,"" ]
&  Y
\end{tikzcd}
\end{center}
 If $D\subseteq g^{-1}(E)_{\rm red} $ is a prime divisor of $Z$, then $D$ is $p$-exceptional, so $D$ is a toric divisor by Lemma \ref{Blanchard}.
\end{proof}

\begin{lemma}\label{lemma:small}
    Let $f:Y\to X$ be a small contraction of normal projective varieties. If $X$ is toric, then so is $Y.$
\end{lemma}
\begin{proof}
    Let $H$ be an ample prime divisor in $Y$. Let $R=\oplus_{m\geq 0}\mathcal{O}_X(mf_*H).$ Since $H=f_*^{-1}f_*H$ is ample, by {\cite[Lemma 6.2]{KM}} $R$ is a sheaf of finitely generated $\mathcal{O}_X$-algebras and $Y=\mathrm{\textbf{Proj}}_X R.$ Since $X$ is toric, there is a toric Weil divisor $D$ in $X$ linearly equivalent to $f_*H$. So, $R\cong \oplus_{m\geq 0}\mathcal{O}_X(mD).$ Since $D$ is a toric divisor, the $n$-torus acts on $R$, hence acts on $Y=\mathrm{\textbf{Proj}}_X R$ extending the action on $X.$ As $f$ is birational, this torus action has a dense orbit. So, $Y$ is toric.
\end{proof}
\begin{lemma} \label{toric flop}
    Let $X$, $X'$, and $X^+$ be normal projective varieties. Consider the following flop diagram: 
    
    \begin{center}
\begin{tikzcd}
X  \arrow[rr,dashed, " " ] \arrow[dr, "\psi"'] & & X^{+} \arrow[dl, "\psi^+"] \\
        & X'
\end{tikzcd}
   \end{center}
    If $X$ is toric image, then there is a smooth proper toric variety $Z$ with surjective morphisms to $X$ and $X^+$, commuting with their morphisms to $X'$. In particular, $X^{+}$ is toric image.
\end{lemma}
\begin{proof}
    Refer to the diagram below for the proof. Suppose $Y$ is a proper toric variety with surjective morphism $p:Y\to X$. Let $Y\xrightarrow{\phi}Y'\xrightarrow{p'}X'$ be the Stein factorization of $\psi\circ p$, so $Y'$ is toric by Lemma \ref{Blanchard}. Let $Y^+$ be the normalization of the irreducible component of $Y'\times_{X'}X^+$ dominating $Y'$ and $X^+$. Let $Z_1$ be the normalization of the irreducible component of $Y\times_{Y'}Y^+$ dominating $Y$ and $Y^+$. The map $Y^+\to Y'$ is a small contraction, so by Lemma \ref{lemma:small}, $Y^+$ is toric. Hence $Y\to Y'$, $Y^+\to Y'$ are toric maps, so $Z_1$ is toric. Now we can take $Z$ to be any toric resolution of singularities of $Z_1.$

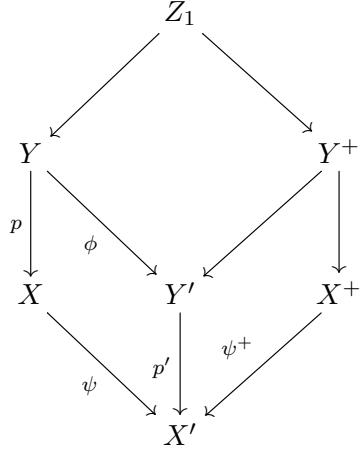
\[
\xymatrix@C=3.5em@R=3.5em{
& Z_1 \ar[dl] \ar[dr] & \\
Y \ar[d]_p \ar[dr]_\phi & & Y^{+} \ar[d] \ar[dl] \\
X \ar[dr]_{\psi} & Y' \ar[d]_{p'} & X^{+} \ar[dl]_{\psi^+} \\
& X' &
}
\]
\end{proof}

\begin{lemma}\label{del Pezzo fibre}
    Let $X$ be a smooth projective variety and let $f:X\to Y$ be a fibration of relative Picard rank $1$ with general fibre $F$ a del Pezzo surface. If $X$ is toric image, then $\deg F=6,8$ or $9.$
\end{lemma}
\begin{proof}
    Since fibres of any toric contraction are toric by Lemma \ref{Blanchard}, a Stein factorization argument shows $F$ is a toric image. As $F$ is a smooth surface, $F$ is toric by {\cite[Proof of Theorem 4.4.1]{achinger2021global}}. Since $F$ is also a del Pezzo surface, and $f$ has relative Picard rank $1$, by \cite{mori1980threefolds}, we have $\deg F=6,8$, or $9.$ 
\end{proof}

\begin{lemma}\label{future}
    If $D$ is a prime divisor in $\mathbb{P}^3$ such that $(\mathbb{P}^3, D)$ is toric image, then $D$ is linear.
\end{lemma}
\begin{proof}
    Follows from \cite[Theorem B]{ksw}.
\end{proof}
\begin{lemma}\label{totaro bott}
    Let $X$ be a normal projective variety and $D$ a reduced Weil divisor on $X$ such that $(X, D)$ is a toric image. Then $(X, D)$ satisfies log Bott vanishing.
\end{lemma}
\begin{proof}
    Follows from \cite[Theorem 5.1]{totaro2023endomorphisms}.
\end{proof}
  Now we are ready to prove Theorem \ref{A}.

  \textit{Proof of Theorem \ref{A}:} Assume $X$ is not Fano, as otherwise we are done by \cite[Theorem 7.2]{totaro2023endomorphisms}. Since projective toric varieties are log Fano, by {\cite[Corollary 5.2]{fujino2012canonical}} $X$ is log Fano. Hence $-K_X$ is big. As we assumed that $-K_X$ is nef, we see that $X$ is weak Fano. As $X$ is not Fano, one of the boundary rays of $\overline{NE}(X)$ is $K_X$-negative, the other one is $K_X$-trivial. Let $\phi:X\to Y$ be the contraction of the $K_X$ negative ray, $\psi:X\to X'$ the contraction of the $K_X$-trivial ray,

  \textbf{Step 1:} We show that one of the following holds:
     \begin{enumerate}
         \item $X$ has a fibration,
         \item $\psi$ is small and $X^+$ has a fibration, where $X^+\rightarrow X'$ is the flop of $\psi$.
     \end{enumerate}

     Here \textit{fibration} means a contraction to a lower dimensional subvariety which is not a point.
     Suppose neither $(1)$ nor $(2)$ holds. We want to get a contradiction. Since $(1)$ does not hold, $\phi$ is birational. By \cite{kollar1991extremal}, $\phi$ is a divisorial contraction, say $E$ is the exceptional divisor.
     
     We define a prime divisor $D\neq E$ in $X$ as follows. If $\psi$ is divisorial, let $D$ be the exceptional divisor of $\psi$. If $\psi$ is small, since $(2)$ does not hold, the $K_{X^+}$-negative contraction is a divisorial contraction, let $D^+$ be the exceptional divisor. Let $D$ be the strict transform of $D^+$ in $X$.
     
     \textbf{Claim 1:} $(X,E+D)$ is toric image.
     
     \begin{proof}
         If $\psi$ is not small, Lemma \ref{lemma:ex} shows $(X,E+D)$ is toric image. Now assume $\psi$ is small.  By Lemma \ref{toric flop}, there is a smooth proper toric variety $Z$ with a commutative diagram
         \begin{center}
\begin{tikzcd}
 & Z \arrow[dl,"r"]   \arrow[dr,"s"]    \\
X  \arrow[rr, dashed, "(\psi^+)^{-1}\circ \psi"] && X^+.
\end{tikzcd}
\end{center}

For a reduced subscheme $W$ of $Z$, let $[W]$ denote the divisorial part of $W$, regarded as a Weil divisor. By Lemma \ref{lemma:ex}, the divisorial parts of $[r^{-1}(E)_{\rm red}]$, $[s^{-1}(D^+)_{\rm red}]$ are toric divisors in $Z$. Since $(\psi^+)^{-1}\circ \psi$ is an isomorphism in codimension $1$, the commutative diagram shows that $[r^{-1}(D)_{\rm red}]-[s^{-1}(D^+)_{\rm red}]$ is an $r$-exceptional divisor, hence a toric divisor. So, $[r^{-1}(D)_{\rm red}]$ is a toric divisor. So, $(X, D+E)$ is a toric image.
     \end{proof}
     
     \textbf{Claim 2:} dim $\phi(E)=0$, $(E,N_{E/X})=(\mathbb{P}^2,\mathcal{O}(-2))$ or $(Q,\mathcal{O}(-1))$, where $Q$ is an irreducible quadric in $\mathbb{P}^3$. Here $N_{E/X}:=\mathcal{O}_X(E)|_E$, which is same as the normal bundle of $E$ in $X$ if $E$ is smooth.
     \begin{proof}
         Suppose dim $\phi(E)=1$. By \cite{kollar1991extremal}, $Y$ is smooth and $\phi$ is the blow up of a smooth curve $C$ in $Y$. $Y$ is smooth toric image threefold of Picard rank $1$, so $Y=\mathbb{P}^3$ by {\cite[Theorem 2]{occhetta2002euler}}. $C$ is a toric image, so $C\cong \mathbb{P}^1 $. We have $(\phi^{-1}\phi(D))_{\rm{red}}\leq D+E$, so by Lemma \ref{push} and Claim 1, $(\mathbb{P}^3, \phi(D))$ is toric image. By Lemma \ref{future}, $\phi(D)=H$ is a hyperplane in $\mathbb{P}^3$. As $D$ and $E$ generate the pseudoeffective cone of $X$, and each fibre of $\phi$ over $C$ intersects $E$ negatively, this fibre must intersect $D$. Hence, $C\subset H$. So, either $C$ is a line, or a degree $2$ planar curve, as $C\cong \mathbb{P}^1 $. In both cases, $X=$ Bl$_C \mathbb{P}^3$ is Fano, a contradiction. So, dim $\phi(E)=0.$

         Say $p=\phi(E).$ So $X=$ Bl$_p(Y)$ by \cite{kollar1991extremal}. If $Y$ is smooth, then $Y=\mathbb{P}^3$ by {\cite[Theorem 2]{occhetta2002euler}}, and $X=$ Bl$_p \mathbb{P}^3$ is Fano, a contradiction. So, $Y$ is singular. By \cite{kollar1991extremal}, $(E,N_{E/X})=(\mathbb{P}^2,\mathcal{O}(-2))$ or $(Q,\mathcal{O}(-1))$, where $Q$ is an irreducible quadric in $\mathbb{P}^3$.
     \end{proof}

     \textbf{Claim 3:} $\psi$ is small birational.
     \begin{proof}
         Suppose not. By {\cite[Corollary 1.5]{jahnke2011threefolds}}, $-K_X$ is base point free. Let $B=\psi(D)$, a smooth curve in $X'.$ As $B$ is toric image, we have $B\cong\mathbb{P}^1.$ Together with Claim 2, it shows $X$ is as in {\cite[No. 1, Table A.5]{jahnke2004threefolds}}. By Corollary \ref{bott}$(6)$, $X$ does not have Bott vanishing, a contradiction to Lemma \ref{totaro bott}.
     \end{proof}

     Let $X^+\xrightarrow{\psi^+}X'$ be the flop of $\psi$. $X^{+}$ is a smooth weak Fano threefold, and by Lemma \ref{toric flop} $X^{+}$ is toric image. By our assumption, $X^{+}$ does not have a fibration. By Claim 2 applied to $X^+$, we see that the $K_{X^+}$-negative extremal contraction $\psi^+:X^+\to Y^+$ is divisorial and contracts $D^+$ to a point, where $D^+$ is the strict transform of $D$. Neither $X$ nor $X^+$ has a fibration, so by {\cite[Proposition 2.5]{jahnke2011threefolds}}, $-K_X, -K_{X^+}$ are base point free. Let $r_X$ be the index of $X$, the largest positive integer dividing $-K_X$ in $\mathrm{Pic }X.$ If $r_X\geq 3$, by {\cite[Proposition 2.12]{jahnke2011threefolds}}, $X$ has a fibration, a contradiction. If $r_X=2$, by {\cite[Theorem 2.13]{jahnke2011threefolds}}, $\phi$ is blow-up of a smooth threefold at a point. This contradicts Claim 2. So, $r_X=1$. By \cite{cutrone2013towards}, $X$ is as in {\cite[Table 8 or 9]{cutrone2013towards}}. By Corollary \ref{bott}, $X$ does not have Bott vanishing, a contradiction to Lemma \ref{totaro bott}.
 
\textbf{Step 2:} We complete the proof.

Replacing $X$ by $X^+$ if necessary, we may assume that $X$ has a fibration. So, $X$ has either a del Pezzo fibration over $\mathbb{P}^1$ or a conic bundle over $\mathbb{P}^2$. 

If $-K_X$ is not spanned, by {\cite[Corollary 1.5]{jahnke2004threefolds}}. So by {\cite[Proposition 2.5]{jahnke2011threefolds}}, $X$ has a degree $1$ del Pezzo fibration, contradicting Lemma \ref{del Pezzo fibre}. So, $-K_X$ is spanned.

If $\phi$ is a del Pezzo fibration, by Corollary \ref{bott} and Lemma \ref{del Pezzo fibre}, general fibre of $\phi$ is $\mathbb{P}^2$. Looking at {\cite[Table 7.1]{jahnke2011threefolds}} and {\cite[Table A.3]{jahnke2004threefolds}}, we see that $X$ is either $\mathbb{P}_{\mathbb{P}^1}(\mathcal{O}_{\mathbb{P}^1}\oplus\mathcal{O}_{\mathbb{P}^1}(1)^2)$ or $\mathbb{P}_{\mathbb{P}^1}(\mathcal{O}_{\mathbb{P}^1}^2\oplus\mathcal{O}_{\mathbb{P}^1}(2))$, so we are done.

Now suppose $\phi$ is a conic bundle. By Corollary \ref{bott}, $X$ is either $\mathbb{P}_{\mathbb{P}^2}(\mathcal{O}_{\mathbb{P}^2}\oplus \mathcal{O}_{\mathbb{P}^2}(3))$ or as in {\cite[No. 1, Table 7.7]{jahnke2011threefolds}}. To rule out the latter, let $Y^+$ be the target of the $K_{X^+}$-negative contraction of $X^+$. Since $X$ is as in {\cite[No. 1, Table 7.7]{jahnke2011threefolds}}, $Y^+$ is smooth threefold of Picard rank $1$, and $-K_{Y^+}^3=54$. Also, by Lemma \ref{toric flop} $X^{+}$ is a toric image, hence so is $Y^{+}$. So by \cite{occhetta2002euler}, $Y^+\cong \mathbb{P}^3$. But then we have $-K_{Y^+}^3=64$, a contradiction.
\begin{remark}
    Note that as a corollary, we get the classification of toric weak Fano $3$-folds of Picard rank $2$ which are not Fano, without using combinatorics of fans.
\end{remark}

\section{Int-amplified endomorphism}
In this section we shall prove Theorem \ref{B}.

Now we are ready to prove Theorem \ref{B}. First we prove the following lemmas.
\begin{lemma}\label{ex1}
    Let $\phi:X\to Y$ be a birational contraction of normal varieties, with $X$ log Fano. Suppose $X$ has an int-amplified endomorphism $f$. Then after replacing $f$ by some power of $f$, the induced dominant rational map $g:Y\to Y$ is an int-amplified endomorphism.
\end{lemma}
\begin{proof}
    This follows from the proof of {\cite[Lemma 6.2]{totaro2023endomorphisms}}.
\end{proof}
\begin{lemma}\label{pol flop}
    Let $X, X^+$ be smooth projective threefolds, and let
    
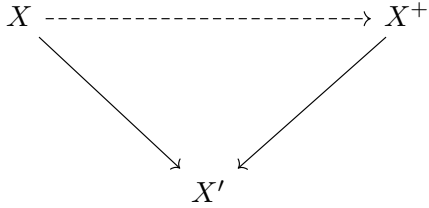
\begin{center}
\begin{tikzcd}
X  \arrow[rr,dashed, " " ] \arrow[dr, ""'] & & X^{+} \arrow[dl, ""] \\
        & X'
\end{tikzcd}
   \end{center} be a flop diagram. Let $D\subset X$ be a prime divisor and $f:X\to X$ an int-amplified endomorphism such that set-theoretically $f^{-1}D=D$ and $f^{-1}C=C$ for all flopping curves $C.$ Then the induced dominant rational map $f^+:X^{+}\dashrightarrow X^{+}$ is an int-amplified endomorphism, and $(f^+)^{-1}D^+=D^+$ set-theoretically, where $D^+$ is the strict transform of $D$ in $X^+$. 
\end{lemma}
\begin{proof}
    The proof is essentially the same as the proof of {\cite[Lemma 6.5]{meng2018building}}. The map $f$ induces an int-amplified endomorphism $f':X'\to X'$ such that $(f')^{-1}D'=D'$, where $D'$ is the image of $D$ in $X'$. By the same proof as in {\cite[Lemma 3.6]{zhang2010polarized}}, $f^+$ is a morphism, hence by {\cite[Theorem 3.3]{meng2020building}}, an int-amplified morphism. The statement  $(f^+)^{-1}D^+=D^+$ is clear.
\end{proof}
\begin{lemma}\label{del Pezzo fibre1}
    Let $X$ be a smooth projective variety and let $\phi:X\to Y$ be a fibration of relative Picard rank $1$ with general fibre $F$ a del Pezzo surface. If $X$ has an int-amplified endomorphism, then $\deg F=6,8$ or $9.$
\end{lemma}
\begin{proof}
    Let $f$ be an int-amplified endomorphism of $X$. By replacing $f$ by a power of $f$, we may assume that $f$ is over an endomorphism $\bar{f}$ of $Y$. For a general point $y_1\in Y$ and $y_2=\bar{f}(y_1)$, we have a surjective endomorphism $f|_{\phi^{-1}(y_1)}: \phi^{-1}(y_1)\to \phi^{-1}(y_2)$, which is not an isomorphism as $f$ is int-amplified. By the same proof as in {\cite[Proposition 4]{beauville2001endomorphisms}}, we have $\deg F\geq 6$. Since $\phi$ has relative Picard rank $1$, by \cite{mori1980threefolds}, we have $\deg F=6,8$ or $9.$ 
\end{proof}
\begin{lemma}\label{totaro bott endo}
    Let $X$ be a normal projective variety and $D$ a reduced Weil divisor on $X$ such that $(X, D)$ has an int-amplified endomorphism. Then $(X, D)$ satisfies log Bott vanishing.
\end{lemma}
\begin{proof}
    Follows from \cite[Theorem 3.1]{totaro2023endomorphisms}.
\end{proof}
  \textit{Proof of Theorem \ref{B}:}  It is very similar to the proof of Theorem \ref{A}. Assume $X$ is not Fano, as otherwise we are done by \cite[Theorem 6.1]{totaro2023endomorphisms}. Since $X$ is weak Fano, it is rationally connected. So, $K_X$ cannot be nef. Let $\phi: X\to Y$ be the contraction of the $K_X$-negative ray of $\overline{NE}(X)$, $\psi: X\to X'$ be the contraction of the $K_X$-trivial ray of $\overline{NE}(X)$. Hence $\psi$ is either divisorial or small birational contraction. If $\psi$ is small, let $\psi^+:X^+\to X'$ be the flop of $\psi$. By {\cite[Proposition 2.2]{jahnke2011threefolds}}, $X^{+}$ is a smooth weak Fano threefold, and by Lemma \ref{pol flop} $X^{+}$ has int-amplified endomorphism.

  \textbf{Step 1:} We show that one of the following holds:
     \begin{enumerate}
         \item $X$ has a fibration,
         \item $\psi$ is small and $X^+$ has a fibration, where $X^+\rightarrow X'$ is the flop of $\psi$.
     \end{enumerate}
     Suppose neither $(1)$ nor $(2)$ holds. We want to get a contradiction. Since $(1)$ does not hold, $\phi$ is birational. By \cite{kollar1991extremal}, $\phi$ is a divisorial contraction, say $E$ is the exceptional divisor. We define a prime divisor $D\neq E$ in $X$ in the same way as in Step 1 of the proof of Theorem \ref{A}.

     \textbf{Claim 1:} $(X,E+D)$ has int-amplified endomorphism.
     
     \begin{proof}

     Let $f$ be an int-amplified endomorphism of $X$. Replacing $f$ by a power of $f$, by Lemmas \ref{pol flop} and \ref{ex1} we have set-theoretically $f^{-1}(D)=D, f^{-1}(E)=E.$
     \end{proof}

      \textbf{Claim 2:} dim $\phi(E)=0$, $(E,N_{E/X})=(\mathbb{P}^2,\mathcal{O}(-2))$ or $(Q,\mathcal{O}(-1))$, where $Q$ is an irreducible quadric in $\mathbb{P}^3$. Here $N_{E/X}:=\mathcal{O}_X(E)|_E$, which is same as the normal bundle of $E$ in $X$ if $E$ is smooth.
     \begin{proof}
       It is almost the same as the proof of Claim 2 in the proof of Theorem \ref{A}.  Suppose dim $\phi(E)=1$. By \cite{kollar1991extremal}, $Y$ is smooth and $\phi$ is the blow up of a smooth curve $C$ in $Y$. $Y$ is smooth Fano threefold of Picard rank $1$ with an int-amplified endomorphism by Lemma \ref{ex1}. So, $Y=\mathbb{P}^3$ by {\cite[Theorem A]{totaro2023bott}}. We have $\phi^{-1}\phi(D)\subset D+E$, so by Claim 1, $(\mathbb{P}^3, \phi(D))$ has int-amplified endomorphism. By {\cite[Corollary 1.2]{horing2017totally}}, $\phi(D)=H$ is a hyperplane in $\mathbb{P}^3$. As $D$ and $E$ generate the pseudoeffective cone of $X$, and each fibre of $\phi$ over $C$ intersects $E$ negatively, this fibre must intersect $D$. Hence, $C\subset H$. So, either $C$ is a line, or a degree $2$ planar curve, or a planar elliptic curve, as $C$ has a non-isomorphic endomorphism. In each case, $X=$ Bl$_C \mathbb{P}^3$ is Fano, a contradiction. So, dim $\phi(E)=0.$

         Say $p=\phi(E).$ So $X=$Bl$_p(Y)$ by \cite{kollar1991extremal}. If $Y$ is smooth, then $Y=\mathbb{P}^3$ by {\cite[Theorem 2]{occhetta2002euler}}, and $X=$ Bl$_p \mathbb{P}^3$ is Fano, a contradiction. So, $Y$ is singular. By \cite{kollar1991extremal}, $(E,N_{E/X})=(\mathbb{P}^2,\mathcal{O}(-2))$ or $(Q,\mathcal{O}(-1))$, where $Q$ is an irreducible quadric in $\mathbb{P}^3$.
     \end{proof}

     \textbf{Claim 3:} $\psi$ is small birational.
     \begin{proof}
         Suppose not.  By {\cite[Corollary 1.5]{jahnke2011threefolds}}, $-K_X$ is base point free. As $B$ has non-isomorphic endomorphism, we have $B\cong\mathbb{P}^1$ or an elliptic curve. Together with Claim 2, it shows $X$ is as in {\cite[No. 1, Table A.5]{jahnke2004threefolds}}. By Corollary \ref{bott}$(6)$, $X$ does not have Bott vanishing, a contradiction to Lemma \ref{totaro bott endo}.
     \end{proof}

     Now we get a contradiction exactly in the same way as in the end of Step 1 of the proof of Theorem \ref{A}.

\textbf{Step 2:} Now we complete the proof exactly in the same way as Step 2 of the proof of Theorem \ref{A}.

\begin{remark}
    A similar proof actually shows a statement more general than Theorem \ref{A}. For a variety $X$ over $\mathbb{C}$, call $X$ to be $F$-liftable, if there is a finitely generated subring $A$ of $\mathbb{C}$, a model $X_A$ of $X$ over $A$, such that for every algebraically closed field $k$ of positive characteristic and ring map $A\to k$, the $k$-scheme $X_A\times_A k$ is $F$-liftable as in \cite[Definition 2.2]{ksw}. We claim that if $X$ is an $F$-liftable weak Fano threefold of Picard rank $2$, then $X$ is toric. This is a special case of {\cite[Conjecture 1]{achinger2021global}}. This statement is a generalisation of Theorem \ref{A}, as by {\cite[Theorem 4.4.1]{achinger2021global}} toric images are F-liftable.

    The proof of this is very similar to the proof of Theorem \ref{B}. By {\cite[Theorem 3.2.4]{achinger2021global}}, an $F$-liftable variety has Bott vanishing. The analogue of Lemma \ref{ex1} is {\cite[Theorem 3.3.6]{achinger2021global}}. The analogue of Lemma \ref{pol flop} follows from the fact that for normal varieties over an algebraically closed field of positive characteristics, $F$-liftability can be checked in codimension $2$ (see \cite[Theorem 3.3.6(b)(iii)]{achinger2021global}). The analogue of Lemma \ref{del Pezzo fibre1} follows from {\cite[Corollary 5.3]{achinger2023global}} and {\cite[Theorem 2]{achinger2021global}}. Rest of the proof is exactly similar to the proof of Theorem \ref{B}.
\end{remark}
\section{Statements and Declarations}

No specific competing financial or non-financial interest to declare.
\printbibliography
\end{document}